\documentclass[psamsfonts]{amsart}  
\usepackage{amssymb}

\usepackage[dvips]{graphicx}

\usepackage{amsmath,amsthm,amscd,amssymb,amsfonts, amsbsy}
\usepackage[mathscr]{eucal}
\usepackage{enumitem}

\usepackage{bbm}

\newcommand{\beq}{\begin{equation}}
\newcommand{\eeq}{\end{equation}}
\newcommand{\bea}{\begin{eqnarray}}
\newcommand{\eea}{\end{eqnarray}}
\newcommand{\beas}{\begin{eqnarray*}}
\newcommand{\eeas}{\end{eqnarray*}}
\newcommand{\expr}[1]{\left( #1 \right)}

\newtheorem{theorem}{Theorem}[section]

\newtheorem{proposition}[theorem]{Proposition}

\newtheorem{lemma}[theorem]{Lemma} 
\newtheorem{remark}[theorem]{Remark}
\newtheorem{example}[theorem]{Example}
\newtheorem{examples}[theorem]{Examples}
\newtheorem{foo}[theorem]{Remarks}

\usepackage[usenames,dvipsnames]{color}

\usepackage[colorlinks,citecolor=red,pagebackref,hypertexnames=false, breaklinks]{hyperref}

\newcommand{\ang}[1]{\left<#1\right>}  

\newcommand{\brak}[1]{\left(#1\right)}    
\newcommand{\crl}[1]{\left\{#1\right\}}   

\newcommand{\N}[1]{\left\|#1\right\|}     
\newcommand{\abs}[1]{\left|#1\right|}     

\newcommand{\bE}{\mathbb E}

\newcommand{\bH}{\mathbb H}
\newcommand{\M}{\mathbb M}
\newcommand{\bN}{\mathbb N}
\newcommand{\bP}{\mathbb P}
\newcommand{\R}{\mathbb R}

\newcommand{\Acal}{\mathcal A}

\newcommand{\F}{\mathcal F}

\newcommand{\Scal}{\mathcal S}
\newcommand{\Tcal}{\mathcal T}

\newcommand{\Y}{\mathcal Y}

\newcommand{\mx}{\mathbf x}
\newcommand{\my}{\mathbf y}

\newcommand{\ld}{\lambda}

\newcommand{\wtvp}{\widetilde\varphi}

\newcommand{\SU}{ \mathbb{ SU}}
\newcommand{\su}{\mathfrak{su}}

\begin{document}

\title[martingale transforms and their projections]{Gundy-Varopoulos martingale transforms  and their projection operators on manifolds and vector bundles}
\author{R. Ba\~nuelos}
\address{Department of Mathematics, Purdue University, West Lafayette, IN 47907, USA}
\email{banuelos@math.purdue.edu}
\author{F. Baudoin} 
\address{Department of Mathematics, University of Connecticut, Storrs, CT 06269}
\email{fabrice.baudoin@uconn.edu}
\author{L. Chen}
\address{Department of Mathematics, University of Connecticut, Storrs, CT 06269}
\email{li.4.chen@uconn.edu}

\thanks{R. Ba\~nuelos supported in part by NSF Grant 1403417-DMS. F. Baudoin supported in part by NSF Grant 1660031-DMS}

\begin{abstract}  This paper proves the $L^p$ boundedness of generalized first order Riesz transforms obtained as conditional expectations of martingale transforms  \`a la  Gundy-Varopoulos for quite general diffusions on manifolds and vector bundles.  Several specific examples and applications are presented: Lie groups of compact type, the Heisenberg group, SU(2),  and Riesz transforms on forms and spinors.
\end{abstract}

\maketitle

\tableofcontents

\section{Introduction}

\subsection{Context}

In Euclidean harmonic analysis, Riesz transforms play a fundamental role in singular integral theory, Hardy space theory and their applications to various areas of analysis and partial differential equations.  The classical Calder\'on-Zygmund theory shows  that Riesz transforms are bounded on $L^p$, for all $1<p<\infty$. Sharp $L^p$ estimates for first order Riesz transforms can be proved using either analytic \cite{IM96} or probabilistic tools  \cite{BW95}.

The study of Riesz transform in different geometric settings has a long history now with a very extensive literature. We would like to mention that Stein \cite{Ste70} introduced Riesz transforms on  compact Lie groups and Strichartz \cite{Str83} first asked whether one could extend their $L^p$ boundedness to complete Riemannian manifolds. Riesz transforms on $k$-forms associated with the Hodge Laplacian were also formulated in   \cite{Str83}.  Ever since, there has been an explosion of literature on this topic using numerous  analytic, geometric, and  probabilistic methods. In the following, we will give a brief overview of previous results. 

 Using the martingale approach via  Littlewood-Paley inequalities as in  
 \cite{Mey84}, 
 Bakry \cite{Bak87} proved that the Riesz transforms are bounded on $L^p$, $1<p<\infty$, on manifolds with non-negative Ricci curvature. An extension was  obtained for the Riesz transform associated with the Hodge Laplacian on $L^p$ spaces of $k$-forms for all $1<p<\infty$ under assumption of Weitzenb\"ock curvature. In a series of papers \cite{LiXD08,LiXD10I,LiXD10II,LiXD14I,LiXD14II},  X. D. Li gave dimension-free estimates  for Riesz transforms in various geometric settings, including  manifolds with assumptions on Bakry-Emery curvature or Weitzenb\"ock curvature, and  K\"ahler manifolds. See also \cite{vanNV17} for  the special case of the   Witten Laplacians.

Another widely-adopted approach to study Riesz transform in various geometric settings is to use the Calder\'on-Zygmund theory for which heat kernel estimates play an essential role.  See, for instance, the setting of Riemannian manifolds in \cite{CD99,LiHQ99, CD03, LiHQ03I, LiHQ03II,ACDH04,AMR08,CMO15,CCFR17} and the references therein, especially  \cite[Section 1.3]{ACDH04} with a quite complete list of previous results,  and the setting of Lie groups, or manifolds with sub-elliptic operators  in \cite{Ale92, CMZ96, BG13,BauB16}. We note that, unlike the martingale approach, the Calder\'on-Zygmund theory do not give constants which are independent of dimension.

There are still many other methods to treat Riesz transforms in different settings. For instance, it was shown in \cite{CCH06} that the Riesz transforms on $n$-dimensional manifolds with Euclidean ends  are bounded on $L^p$ if and only if $p\in (1,n)$. This result was generalized in \cite{Car07, Dev14} and in \cite{GS15} to $k$-forms for asymptotically conic manifolds. See also \cite{MPR15} for Riesz transforms of $k$-forms on the Heisenberg group. Last but not the least, we refer to  \cite{CarbD13} for Bellman function techniques to study Riesz transforms on manifolds under Bakry-\'Emery type curvature assumptions.

\subsection{Main results}

In the present paper, we aim to study a class of operators under very general setting which are projections of martingale transforms in the style of the now classical  Gundy-Varopoulos construction on $\R^d$. We are interested in the near sharp estimates or at least  dimension free estimates for these operators. More precisely, let $\M$ be a smooth manifold with a smooth measure $\mu$. Let $X_1,\cdots,X_d$ be locally Lipschitz vector fields defined on $\M$. We consider the Schr\"odinger operator 
\[
L=-\sum_{i=1}^d X_i^* X_i+V,
\]
where the $X_i^*$ denotes the formal adjoint of $X_i$ with respect to $\mu$ and where $V:\M\to \R$ is a non-positive smooth potential. Assume that $L$ is essentially self-adjoint on the space $\mathcal{S}(\M)$ of smooth and rapidly decreasing functions. 
Denote by $P_y=e^{-y \sqrt{-L}}$ the Poisson semigroup. For any $1\le i\le d$, we consider the operator
\[
T_if=\int_0^{+\infty} y P_y \brak{\sqrt{-L} X_i-X_i^* \sqrt{-L}} P_y fdy.
\]

Our main result is the following.
\begin{theorem}\label{thm:MT}
Fix $1<p<\infty$ and set $p^*=\max\{p,\frac{p}{p-1}\}$. Then for every $f\in \mathcal{S}(\M)$,
\begin{equation}\label{TA-norm1}
\|T_i f\|_p \le \left(\frac{3}{2}\right) (p^*-1)\|f\|_p.
\end{equation}

If the potential $V\equiv0$, then
\begin{equation}\label{TA-norm2}
\|T_i f\|_p \le \frac{1}{2}\cot\!\expr{\frac{\pi}{2 p^*}} \|f\|_p.
\end{equation}
\end{theorem}

It is interesting to  note here that the constant in \eqref{TA-norm2}, up to the factor $\frac{1}{2}$, is the sharp $L^p$ constant for the classical Hilbert transform in $\R$ and the first order Riesz transforms on $\R^d$; see \cite{MR0312140}, \cite{IM96}, \cite{BW95}.  
 
Applications of Theorem \ref{thm:MT} (and of the slightly more general Theorem \ref{thm:MT2} given below) include, but are not limited to, the following examples.
\begin{proposition}
Let $G$ be a  Lie group of compact type endowed with a bi-invariant Riemannian structure. Then
\[
\N{X_i(\sqrt{-L})^{-1} f}_{p} \le \cot\!\expr{\frac{\pi}{2 p^*}}  \N{f}_p.  
\]
\end{proposition}
The constant in the previous example is sharp and was first obtained in  \cite{MR1650585}.

\begin{proposition}
Let $\bH^n$ be the Heisenberg group with left invariant vector fields 
$X_1,\cdots,X_n,Y_1,\cdots, Y_n, Z$. Consider the sublaplacian $L=\sum_{j=1}^n (X_j^2+Y_j^2)$. Then
\[
\N{[W_j,\sqrt{-L}] f}_{p} \le  \sqrt{2}(p^*-1) \N{Zf}_p,
\]
where $W_j=X_j+iY_j$ is the  complex gradient.
\end{proposition}

\begin{proposition}
Let  $X,Y, Z$ be the left invariant vector fields on $\SU(2)$ and consider the sublaplacian $L=X^2+Y^2$. Then
\[
\N{[W,\sqrt{-L}] f}_{p} \le  2\sqrt{2}(p^*-1)  \N{(iZ+1)f}_p,
\]
where $W=X+iY$ is the  complex gradient.
\end{proposition}

Theorem 1.1 can also be generalized to the framework of vector bundles under general assumptions, which allows us to obtain dimension-free estimates for Riesz transforms on very general vector bundles and yields a generalization and simplification of the results in \cite{LiXD08,LiXD10I,LiXD10II,LiXD14I,LiXD14II}.
Let $\mathcal{E}$ be a finite-dimensional vector
bundle over $\mathbb{M}$. We denote by $\Gamma ( \mathbb{M},
\mathcal{E} )$ the space of smooth  sections of this bundle. Let now $\nabla$
denote a metric connection on $\mathcal{E}$. We consider an operator on $\Gamma ( \mathbb{M},
\mathcal{E} )$  that can be written as
\[
\mathcal{L}=\mathcal{F}+\nabla_{0}+ \sum_{i=1}^d \nabla_{i}^2,
\]
where
\[
\nabla_i=\nabla_{X_i}, \quad 0 \le i \le d,
\]
and the $X_i$'s  are smooth vector fields on $\mathbb{M}$ and 
$\mathcal{F}$  is a smooth symmetric and non positive potential (that is a smooth section of the
bundle $\mathbf{End}(\mathcal{E})$). We will assume that $\mathcal{L}$ is non-positive and essentially self-adjoint on the space $\Gamma_0(\M,\mathcal{E})$ of smooth and compactly supported sections. We consider then a first order differential operator $d_a$ on $\Gamma ( \mathbb{M},
\mathcal{E} )$ that can be written as
\[
d_a=\sum_{i=1}^d a_i \nabla_{X_i},
\]
where $a_1,\cdots,a_d$ are smooth sections of the bundle $\mathbf{End}(\mathcal{E})$. Our main assumptions  are that 
\[
d_a \mathcal{L} \eta =\mathcal{L} d_a \eta,\quad \eta \in \Gamma ( \mathbb{M},
\mathcal{E} ),
\]
and that 
\[
\| d_a \eta \|^2 \le C \sum_{i=1}^d \| \nabla_{X_i} \eta \|^2, \quad \eta \in \Gamma ( \mathbb{M},
\mathcal{E} ),
\]
for some constant $C\ge 0$.  Then we have  the following:

\begin{theorem} 
For $1<p<\infty$,
\[
\| d_a (-\mathcal{L})^{-1/2} \eta \|_p \le 6C(p^*-1)\| \eta \|_p.
\]
\end{theorem}

This result can be applied on the following examples:
\begin{enumerate}
\item On Riemannian manifolds with non-negative Weitzenb\"ock curvature. We recover then  the boundedness Riesz transform on forms associated with Hodge-de Rham Laplacian as obtained by X.D. Li \cite{LiXD08}.
\item On Riemannian manifolds which have non-negative scalar curvature and admit a spin structure. We prove then the boundedness of the Riesz transforms on spinors associated with the Dirac operators, which is a new result.
\end{enumerate}

\

\noindent\textbf{Convention:} \textit{Throughout the paper,  the diffusion operators we consider ($L$ and $\mathcal{L}$) will be assumed to have no spectral gap. In the case, where those operators do have a spectral gap, all the results hold when restricting $L^p(\M,\mu)$ to the space $L_0^p(\M,\mu)=\{ f \in L^p(\M,\mu), \int f d\mu =0\}$.}

\section{Scalar operators constructed from martingale transforms}

\subsection{A general theorem}
Let $\M$ be a smooth and complete Riemannian manifold with a smooth measure $\mu$. Let $X_1,\cdots,X_d$ be locally Lipschitz vector fields defined on $\M$. We consider the Schr\"odinger operator 
\[
L=-\sum_{i=1}^d X_i^* X_i+V,
\]
where the $X_i^*$ denotes the formal adjoint of $X_i$ with respect to $\mu$ and where $V:\M\to \R$ is a non-positive smooth potential. Assume that $L$ is essentially self-adjoint with respect to $\mu$ on the space $\mathcal{S}(\M)$ of smooth and rapidly decreasing functions. We denote by $(H_t)_{t\ge 0}$ the heat semigroup with generator $L$. We can write   
\[
L=\sum_{i=1}^d X_i^2+X_0+V,
\]
for some locally Lipchitz vector field $X_0$.

Let $(Y_t)_{t\ge 0}$ be the diffusion process on $\M$ with generator  $\sum_{i=1}^d X_i^2+X_0$ starting from the distribution $\mu$. We assume that $(Y_t)_{t\ge 0}$ is non explosive,  it can then be constructed  via the Stratonovitch stochastic differential equation
\[
dY_t=X_0(t)dt+\sum_{i=1}^d X_i(Y_t)\circ d\beta_t^i, 
\]
where $\beta_t=(\beta_t^1,\cdots,\beta_t^d)$ is the Brownian motion on $\R^d$ with generator $\sum_{i=1}^d \frac{\partial^2}{\partial x_i^2}$.

Let $(B_t)_{t\ge 0}$ be a one-dimensional Brownian motion on $\R$, with generator $\frac{d^2}{dy^2}$ starting from $y_0>0$, which is independent from $(Y_t)_{t\ge 0}$. Then $\bE(B_t^2)=2t$. Set the stopping time
$$
\tau=\inf\{t>0, \,B_t=0\}.
$$

We denote by $H_t=e^{tL}$ the heat semigroup and by $P_y=e^{-y \sqrt{-L}}$ the Poisson semigroup. 
By the Feynman-Kac formula, the heat semigroup $H_t$ acting on $f\in \mathcal{S}(\M)$ can be written as 
\[
H_tf(x)=\bE_x\brak{e^{\int_0^t V(Y_s)ds} f(Y_t)}.
\]
To work on a set of test functions which is large enough, we assume that $\mathcal{S}(\M)$ is stable by $H_t, P_y,  X_i,X_i^*$ and $\sqrt{-L}$. Let $1\le i\le d$. Consider the operator
\[
T_i=\int_0^{+\infty} y P_y \brak{\sqrt{-L} X_i-X_i^* \sqrt{-L}} P_y dy.
\]
We assume that $T_i$ is pointwisely well defined on $\mathcal{S}(\M)$ and $T_i(\mathcal{S}(\M)) \subset \mathcal{S}(\M)$. We have then the following Gundy-Varopoulos type probabilistic representation of $T_i$.

\begin{lemma}\label{lem:CE}
Let $f\in \mathcal{S}(\M)$. For almost all $x\in \M$, we have, for $i=1, 2, \dots, d$,  
\[
T_i f(x)=- \frac{1}{2} \lim_{y_0\to \infty} \bE_{y_0} \brak{ e^{\int_0^{\tau} V(Y_v)dv }\int_0^{\tau}  e^{-\int_0^{s} V(Y_v)dv}A_i(\nabla,\partial_y)^{\mathrm T} Q^V f(Y_s, B_s) (d\beta_s,dB_s)\mid Y_{\tau}=x},
\]
where  $\nabla=(X_1,\cdots, X_d)$, and $A_i$ is a $(d+1)\times (d+1)$ matrix with $a_{i(d+1)}=-1$, $a_{(d+1)i}=1$ and otherwise $0$.
\end{lemma}

To prove the lemma, we introduce some notation.
For $f\in \mathcal{S}(\M)$, we denote 
\[
Q^Vf(x,y)=P_y f(x)=e^{-y\sqrt{-L}} f(x),
\]
and 
\[
M_t^f= e^{\int_0^{t\wedge \tau} V(Y_s)ds} Q^Vf(Y_{t\wedge \tau},B_{t\wedge \tau}).
\]

\begin{lemma}\label{martingale:lk}
Let $f\in \mathcal{S}(\M)$. The process $M_t^f$ is a square integrable martingale. Moreover, the quadratic variation is
\begin{eqnarray*}
\ang{M^f}_t&=&2 \int_0^{t\wedge \tau}\brak{ e^{\int_0^{\tau} V(Y_v)dv} \partial_y Q^Vf(Y_s,B_s)}^2ds\\
& &+2\int_0^{t\wedge \tau} \sum_{i=1}^d \brak{e^{\int_0^{\tau} V(Y_v)dv}X_i Q^Vf(Y_s,B_s)}^2 ds.  
\end{eqnarray*}
\end{lemma}
\begin{proof}
We follow the proof in \cite{Mey76}. Note that 
\[
M_{\tau}^f=e^{\int_0^{\tau} V(Y_v)dv} Q^Vf(Y_{\tau},0) =e^{\int_0^{\tau} V(Y_u)du} f(Y_{\tau}).
\]
Let $\ld_u$ be the distribution of $\tau$ at $B_0=u$. Since the processes $Y_t$ and $B_t$ are independent, then by the Feynman-Kac formula,
\begin{align*}
\bE\brak{e^{\int_0^{\tau} V(Y_v)dv} f(Y_{\tau})\mid Y_0=x, B_0=u}
&=\int_{0}^{\infty} \bE\brak{e^{\int_0^{s} V(Y_v)dv} f(Y_{s})\mid Y_0=x} \ld_u(ds)
\\ &= \int_{0}^{\infty} H_s f(y) \lambda_u(ds)=Q^Vf(x,u).
\end{align*}
The last equality above follows from the subordination formula due to the fact that 
\[
\ld_u(ds)=\frac{u}{2\sqrt \pi} e^{-u^2/4s}s^{-3/2} ds.
\]
Therefore we have
\begin{align*}
\bE\brak{M_{\tau}^f\mid \F_s} &=\bE\brak{e^{\int_0^{\tau} V(Y_v)dv}f(Y_{\tau}) \mid \F_s}
\\ &=
\bE\brak{e^{\int_0^{\tau} V(Y_v)dv}f(Y_{\tau})1_{\tau\le s} \mid\F_s}+\bE\brak{e^{\int_0^{\tau} V(Y_v)dv}f(Y_{\tau})1_{\tau>s} \mid\F_s}
\\ &=
e^{\int_0^{\tau} V(Y_v)dv}f(Y_{\tau})1_{\tau\le s}+\bE\brak{e^{\int_0^{\tau} V(Y_v)dv}f(Y_{\tau})1_{\tau>s} \mid(Y_s,B_s)}
=M_s^f.
\end{align*}
Now applying It\^o's formula, we obtain that for any $0<t\le \tau$,
\begin{align}\label{eq:ito}
e^{\int_0^t V(Y_v)dv} Q^Vf(Y_t,B_t)=& Q^Vf(Y_0,B_0)+\sum_{i=1}^d \int_0^\tau e^{\int_0^s V(Y_v)dv} X_i Q^Vf(Y_s,B_s) d\beta_s^i
\\&+\int_0^\tau e^{\int_0^s V(Y_v)dv} \partial_yQ^Vf(Y_s,B_s) dB_s.\nonumber
\end{align}
Hence
\begin{eqnarray*}
\langle M^f\rangle_t &=&2 \int_0^{t\wedge \tau}\brak{ e^{\int_0^{\tau} V(Y_v)dv} \partial_y Q^Vf(Y_s,B_s)}^2ds\\
& &+2\int_0^{t\wedge \tau} \sum_{i=1}^d \brak{e^{\int_0^{\tau} V(Y_v)dv}X_i Q^Vf(Y_s,B_s)}^2 ds.  
\end{eqnarray*}
\end{proof}

\begin{proof}[Proof of Lemma \ref{lem:CE}]
Let $f\in \mathcal{S}(\M)$ and $g\in \mathcal{S}(\M)$. Note that $Y$ and $B$ are independent and recall that 
\[
M_{\tau}^g=e^{\int_0^{\tau} V(Y_v)dv }g(Y_{\tau}).
\]
By It\^o's formula \eqref{eq:ito} and the It\^o isometry, we have
\begin{align*}
&\int_{\M} g(x)\brak{ e^{\int_0^{\tau} V(Y_v)dv }\int_0^{\tau}  e^{-\int_0^{s} V(Y_v)dv}A_i (\nabla,\partial_y)^{\mathrm T} Q^V f(Y_s, B_s) (d\beta_s,dB_s)\mid Y_{\tau}=x} d\mu(x)
\\&=
\bE_{y_0} \brak{g(Y_{\tau}) e^{\int_0^{\tau} V(Y_v)dv }\int_0^{\tau}  e^{-\int_0^{s} V(Y_v)dv} A_i (\nabla,\partial_y)^{\mathrm T} Q^V f(Y_s, B_s) (d\beta_s,dB_s)}
\\&=
2 \bE_{y_0} \brak{\int_0^{\tau} \partial_y Q^Vg(Y_s,B_s) X_iQ^V f(Y_s, B_s) ds-\int_0^{\tau}  X_iQ^Vg(Y_s,B_s)\partial_y Q^V f(Y_s, B_s) ds}
\\ &=
2\int_{\M} \int_0^{\infty} (y_0 \wedge y) \brak{\partial_y Q^Vg(x,y) X_iQ^V f(x,y)-X_i Q^Vg(x,y) \partial_yQ^V f(x,y)} dy d\mu(x).
\end{align*}
The last equality follows from the facts that the Green function of the Brownian motion is killed at $0$ and $Y_s$ is distributed according to $\mu$.

Since $L$ is self-adjoint, then
\[
\int_{\M} \int_0^{\infty} y\partial_y Q^Vg(x,y) X_iQ^V f(x,y) dy d\mu(x)= -\int_{\M} g(x) \int_0^{\infty} y P_y \sqrt{-L}X_i P_y f(x) dy d\mu(x),
\]
and
\[
\int_{\M} \int_0^{\infty} yX_i Q^Vg(x,y) \partial_y Q^V f(x,y) dy d\mu(x)= -\int_{\M} g(x) \int_0^{\infty} y P_y X_i^* \sqrt{-L}P_y f(x) dy d\mu(x).
\]
Therefore we obtain
\begin{align*}
T_i f(x)=-\frac{1}{2}\lim_{y_0\to \infty} \bE_{y_0} \brak{e^{\int_0^{\tau} V(Y_v)dv }\int_0^{\tau}  e^{-\int_0^{s} V(Y_v)dv} A_i(\nabla,\partial_y)^{\mathrm T}Q^V f(Y_s, B_s) (d\beta_s,dB_s)\,\Big|\, Y_{\tau}=x}.
\end{align*}
\end{proof}

To study the $L^p$ boundedness of $T_i$, we shall use 
some extensions of  Burkholder's celebrated theorem on martingale transforms.  Before stating the results, we introduce some necessary probabilistic background.
Suppose that $(\Omega, \F, \bP)$ is a complete probability space, filtered by $\F=\{\F_t\}_{t\ge 0}$, a family of right continuous sub-$\sigma$-fields of $\F$. Assume that $\F_0$ contains all the events of probability $0$. Let $X$ and $Y$ be adapted, real-valued martingales which have right-continuous paths with left-limits (r.c.l.l.).  The martingale $Y$ is differentially subordinate to $X$ if $|Y_0|\le |X_0|$ and $\ang{X}_t-\ang{Y}_t$ is a nondecreasing and nonnegative function of $t$.  The martingales $X_t$ and $Y_t$ are said to be orthogonal if  the covariation process $\ang{X, Y}_t=0$ for all $t$. Ba\~nuelos and Wang \cite{BW95} proved the following sharp inequality extending the classical results of Burkholder.  We always assume the martingale $X$ (hence $Y$) is $L^p$ bounded for $1<p<\infty$ and by $X$ in he inequalities below we mean $X_{\infty}$.  Similarly for $Y$.  
\begin{theorem}\label{thm:BW95}
Let $X$ and $Y$ be two martingales with continuous paths such that $Y$ is differentially subordinate to $X$. Fix  $1<p<\infty$ and set $p^*=\max\{p,\frac{p}{p-1}\}$.  Then 
\begin{equation}\label{eq:MT}
\|Y\|_p \le (p^*-1)\|X\|_p.
\end{equation}
Furthermore, suppose the martingales $X$ and $Y$ are orthogonal.  Then 
\begin{equation}\label{eq:OMT}
\|Y\|_p \le \cot\!\expr{\frac{\pi}{2 p^*}}\|X\|_p. 
 \end{equation}
Both of these inequalities are sharp.

\end{theorem}
More generally, Ba\~nuelos and Os{\c e}kowski \cite{BO15} proved that 
\begin{theorem}\label{thm:BO15}
Let $X$ and $Y$ be two martingales with continuous paths such that $Y$ is differentially subordinate to $X$.
Consider the process 
\[
Z_t=e^{\int_0^{t} V_sds }\int_0^t e^{-\int_0^{s} V_vdv }dY_s,
\]
where $(V_t)_{t\ge 0}$ is a non-positive adapted and continuous process. 
For $1<p<\infty$, we  have the sharp bound
\begin{equation}\label{MT}
\|Z\|_p \le (p^*-1)\|X\|_p.
\end{equation}
\end{theorem}

\begin{proof}[Proof of Theorem  \ref{thm:MT}]
We first consider the case that $V\equiv 0$. Note that the martingale 
\[
N_t=\int_0^{t\wedge \tau}  A_i(\nabla , \partial_y)^{\mathrm T} Q f(Y_s, B_s) (d\beta_s,dB_s)
\] 
is differentially subordinate to 
$M_t^f=Qf(Y_{t\wedge \tau},B_{t\wedge \tau})$. In addition, since the matrix $A_i$ is orthogonal, that is, $\ang{Av, v}=0$ for all $v\in \R^{d+1}$, $\ang{M^f, N}_t=0$.  Hence Lemma \ref{lem:CE} and Theorem \ref{thm:BW95} gives us $$\|T_i f\|_p \le   {\frac{1}{2}}\cot\!\expr{\frac{\pi}{2 p^*}}\|f\|_p.$$

Next we deal with the case $V\neq 0$. The stochastic integral $$\int_0^{t\wedge \tau}  A_i (\nabla,\partial_y)^{\mathrm T} Q^V f(Y_s, B_s) (d\beta_s,dB_s)$$  is subordinate to $$\int_0^{t\wedge \tau} (\nabla,\partial_y)^{\mathrm T} Q^V f(Y_s, B_s) (d\beta_s,dB_s).$$

Using It\^o's formula for $Q^Vf(Y_t,B_t)$, we have
\begin{align*}
Q^Vf(Y_{t\wedge \tau},B_{t\wedge \tau})=&Q^Vf(Y_0,B_0)+\int_0^{t\wedge \tau} (\nabla,\partial_y)^{\mathrm T} Q^V f(Y_s, B_s) (d\beta_s,dB_s)
\\ &+2\int_0^{t\wedge \tau} \brak{\partial_y^2+\sum_{i=1}^dX_i^2+X_0} Q^V f(Y_s, B_s) ds.
\end{align*}
Since $Q^Vf(x,y)=e^{-y\sqrt{-L}}f(x)$ satisfies 
\[
\brak{\partial_y^2+\sum_{i=1}^dX_i^2+X_0} Q^V f=-VQ^V f,
\]
then we get 
\begin{eqnarray*}
Q^Vf(Y_{t\wedge \tau},B_{t\wedge \tau})=Q^Vf(Y_0,B_0)&+&\int_0^{t\wedge \tau} (\nabla,\partial_y)^{\mathrm T} Q^V f(Y_s, B_s) (d\beta_s,dB_s)\\
&-&2\int_0^{t\wedge \tau} V(Y_s) Q^V f(Y_s, B_s) ds.
\end{eqnarray*}
Suppose  $f\ge 0$. Then $Q^Vf(Y_{t\wedge \tau},B_{t\wedge \tau})$ is a non-negative submartingale. It follows from Lenglart-L\'epingle-Pratelli \cite[Theorem 3.2, part 3)]{MR580107}
 that 
\begin{eqnarray*}
\N{Q^Vf(Y_0,B_0)-2\int_0^{\tau} V(Y_s) Q^V f(Y_s, B_s) ds}_p &\le& p\N{Q^V f(Y_\tau, B_\tau)}\\
&=&p\N{f(Y_\tau)}_p= p\N{f}_p.
\end{eqnarray*}

This yields
\begin{equation}\label{submar|f|}
\N{\int_0^{\tau} (\nabla,\partial_y)^{\mathrm T} Q^V f(Y_s, B_s) (d\beta_s,dB_s)} \le {(p+1)}\N{f}_p.
\end{equation}
For a general $f$, write
\[
A_t f:=Q^Vf(Y_0,B_0)-2\int_0^{t\wedge \tau} V(Y_s) Q^V f(Y_s, B_s) ds.
\]
Notice that since $|Q^Vf|\le Q^V|f|$ and $V$ is non-positive, then we have $|A_t f|\le A_t|f|$ and the above argument shows that \eqref{submar|f|} holds for general  $f$. 

We now assume that $1<p\leq 2$.  Applying Lemma \ref{lem:CE} and Theorem \ref{thm:BO15}, we conclude that
\begin{equation}\label{p<2}
\N{T_if}_p \le 3\left(\frac{p^*-1}{2}\right)\N{f}_p, \,\,\,\,\,\, 1<p\leq 2.  
\end{equation}

To deal with the case of $2\leq p<\infty$, we recall (as in \cite{BB13}), that if $T_A$ is the operator constructed as above with $A_i$ replaced by a general $(d+1)\times (d+1)$ matrix $A$, its adjoint, $T_A^{*}$ is given by $T_A*$.  Thus, by duality, using the fact that $(p^*-1)$  equals $(p-1)$, for $2\leq p<\infty$, and equals $\frac{1}{(p-1)}$, for $1<p\leq 2$, we get

\begin{equation}\label{p>2}
\N{T_if}_p \le 3\left(\frac{p^*-1}{2}\right)\N{f}_p, \,\,\,\,\,\, 2<p\leq \infty.  
\end{equation}

The estimates \eqref{p<2} and \eqref{p>2} give the estimate \eqref{TA-norm1}.

\end{proof}
The same proof above gives the following more general result for conditional expectations (projections) of martingale transforms as above. 
\begin{theorem}\label{thm:MT2}
Let $f\in \mathcal{S}(\M)$ and let $A$ be a $(d+1)\times (d+1)$ matrix of norm $\|A\|$.  Then for $1<p<\infty$, 
\[
T_A f(x)= -\frac{1}{2}\lim_{y_0\to \infty} \bE_{y_0} \brak{ e^{\int_0^{\tau} V(Y_v)dv }\int_0^{\tau}  e^{-\int_0^{s} V(Y_v)dv}A(\nabla,\partial_y)^{\mathrm T} Q^V f(Y_s, B_s) (d\beta_s,dB_s)\mid Y_{\tau}=x},
\]
satisfies 
\begin{equation}\label{TA-norm3} 
\|T_A f\|_p \le \|A\| \left(\frac{3}{2}\right) (p^*-1) \|f\|_p, 
\end{equation}
and 
if $V\equiv0$, 
\begin{equation}\label{TA-norm4}
\|T_A f\|_p \le\|A\|\frac{(p^*-1)}{2} \|f\|_p. 
\end{equation}
\end{theorem}

\subsection{Example 1. Lie groups of compact type}
Let $G$ be a Lie group of compact type with Lie algebra $\mathfrak{g}$. We endow $G$ with a bi-invariant Riemannian structure and consider an orthonormal basis $X_1,\cdots, X_d$ of $\mathfrak{g}$.  In this setting the Laplace-Beltrami operator can be written as 
\[
L= \sum_{i=1}^d X_i^2.
\]
It is essentially self-adjoint on the space of smooth and compactly supported functions. Then $X_i^*=-X_i$ and $X_i$ commutes with $P_y$. We easily see that
\[
T_i=\int_0^{+\infty} y P_y \brak{\sqrt{-L} X_i-X_i^* \sqrt{-L}} P_y dy
=2X_i \sqrt{-L} \int_0^{+\infty} y P_y P_y dy=\frac{1}{2} X_i (\sqrt{-L})^{-1}.
\]
As a consequence of Lemma \ref{lem:CE} and Theorem  \ref{thm:MT}, the Riesz transform is bounded on $L^p(M)$ and we have the estimate
\[
\N{X_i (\sqrt{-L})^{-1}}_{L^p\to L^p} \le  \cot\!\expr{\frac{\pi}{2 p^*}}.
\]
This inequality was first proved  in  \cite{MR1650585} where the proof is also based on the martingale inequality \eqref{eq:OMT}

\subsection{Example 2. Heisenberg group}

Another interesting example is given by the Heisenberg group. The Heisenberg group is the set 
\[
\bH^n=\{(x,y,z): x\in \R^n, y\in \R^n, z\in \R\}
\]
 endowed with the group law
\[
(x,y,z) \cdot (x',y',z')=\brak{x+x',y+y',z+z'+\frac12\brak{\ang{x,y'}_{\R^n}-\ang{y,x'}_{\R^n}}}.
\]
Consider the left-invariant vector fields: for any $j \in \bN$, $1\le j\le n$,
\[
X_j=\partial_{x_j}-\frac{y_j}{2}\partial_z,\quad Y_j=\partial_{y_j}+\frac{x_j}{2}\partial_z, \quad Z=\partial_z,
\]
and the sublaplacian
\[
L=\sum_{j=1}^n\brak{X_j^2+Y_j^2}.
\]

Denote by $\Scal(\bH^n)$ ($=\Scal (\R^{2n+1})$) the Schwartz space of smooth rapidly decreasing functions on the Heisenberg group. Equivalently, 
\[
\Scal(\bH^n)=\crl{f\in C^{\infty}(\bH^n):\sup_{\mathbf x\in \bH }(1+|\mathbf x|)^q \abs{X^K f(\mx)}<\infty, \forall K \in (\bN^*)^{2n}, \forall q\in \bN}
\]
where $|\mx|=((x^2+y^2)^2+z^2)^{1/4}$, $\bN^*=\bN\cup \{0\}$ and 
\[
X^K=X_1^{k_1}\cdots X_n^{k_n} Y_1^{k_{n+1}} \cdots Y_n^{k_{2n}} , \quad \text{where } K=(k_1,k_2,\cdots, k_{2n}) \in (\bN^*)^{2n}.
\]
Let $d(\mx,\my)$ be the Carnot-Carath\'eodory distance. Notice that $|\mx|\simeq d(0,\mx)$.

The sublaplacian $L$ is essentially self-adjoint on $\Scal(\bH^n)$. Denote by $[U,V]=UV-VU$ the commutator of $U$ and $V$, then for any $1\le j,\, k\le n$,
\[
[X_j,Y_k]=\delta_{jk}Z, \quad [X_j,Z]=0, \quad [Y_k,Z]=0.
\]
Let $W_j=X_j+iY_j$ be the complex gradient, then
\begin{equation}\label{eq:WL}
W_jL=(L-2iZ)W_j.
\end{equation}
In other words, we have $[W_j,L]=-2iZW_j$. Note also that $[W_j,Z]=0$ and  $[L,Z]=0$.

The right invariant vector fields are given by
\[
\hat X_j=\partial_{x_j}+\frac{y_j}{2}\partial_z,\quad \hat Y_j=\partial_{y_j}-\frac{x_j}{2}\partial_z, \quad \hat Z=\partial_z.
 \]
Then we have the right-invariant sublaplacian and complex gradient, denoted by $\hat L$ and $\hat W_j$, respectively.

Let $(H_t)_{t>0}=(e^{t L})_{t>0}$  be the heat semigroup generated by $L$ and $h_t(\mx)$ be the corresponding heat kernel at $0$. Hence $h_t(\mx,\my)=h_t(\mx \my^{-1})$. Let  $(P_t)_{t>0}=(e^{-t\sqrt{-L}})_{t>0}$ be the Poisson semigroup and $p_t$ be the Poisson kernel at $0$. Similarly, let $(\hat H_t)_{t>0}=(e^{t\hat L})_{t>0}$ and $(\hat P_t)_{t>0}=(e^{-t\sqrt{-\hat L}})_{t>0}$ be the heat and Poisson semigroups generated by $\hat L$. The corresponding heat and Poisson kernels are denoted by $\hat h_t$ and $\hat p_t$.  

We have that  $\Scal(\bH^n)$ is left globally stable by $L$ and by $H_t$ for any $t\ge 0$ (see \cite[Lemma 2.1]{BBBC}).
 Moreover, define  $\sqrt {-L}f$ via the heat semigroup as follows:
\begin{equation}\label{eq:sqrtL}
\sqrt {-L}f(\mathbf x) =-\frac{1}{2\sqrt{\pi}} \int_0^{\infty} t^{-3/2} (H_tf(\mathbf x)-f(\mathbf x)) dt.
\end{equation}
Then
\begin{lemma}\label{lem:Schwartz}
For any $f\in \Scal (\bH^n)$, we have $\sqrt{-L} f\in \Scal (\bH^n)$.
\end{lemma}

\begin{proof}
Let $f\in \Scal (\bH^n)$, then for any $K=(k_1,k_2,\cdots, k_{2n}) \in (\bN^*)^{2n}$ and $q\in \bN$,
\[
\abs{\hat X^K f(\mx)} \le \frac{C}{(1+d^2(0,\mx))^q},
\]
where $\hat X^K=\hat X_1^{k_1}\cdots \hat X_n^{k_n} \hat Y_1^{k_{n+1}} \cdots \hat Y_n^{k_{2n}}$.

Following the argument in \cite{BBBC}, we have  that for any $t\ge 1$
\begin{align*}
\abs{ X^K  H_tf(0)}&= \abs{H_t \hat X^K  f(0)}
\le H_t \brak{\frac{C}{(1+d^2(0,\cdot))^q}}(0) \le \bE_0 \brak{\frac{C}{(1+d^2(0,X_t))^q}}\nonumber
\\ &\le \nonumber
\bE_0 \brak{\frac{C}{(1+td^2(0,X_1))^q}} \le 
\frac{C}{t^q}\bE_0 \brak{\frac{1}{(1+d^2(0,X_1))^q}} 
\\ &\le 
\frac{C}{t^q}H_1\brak{\frac{C}{(1+d^2(0,\cdot))^q}}(0).
\end{align*}

By the left invariance, we obtain that 
\begin{equation}\label{eq:large t}
\abs{ X^K  H_tf(x)}= \abs{H_t \hat X^K  f(0)} \le \frac{C}{t^q} (1+d^2(0,\mx))^{-q}
\end{equation}
Also recall that for $0<t<1$, 
\begin{align}\label{eq:small t}
\abs{ X^K  H_tf}&= \abs{H_t \hat X^K  f}
 \le e^{C t} (1+d^2(0,\mx))^{-q}.
\end{align}

Ignoring the constant, we  rewrite \eqref{eq:sqrtL} as 
\[
\sqrt {-L}f(\mathbf x)=  \int_0^{1} t^{-3/2} \int_0^t LH_sf(\mathbf x)ds dt+\int_{1}^{\infty} t^{-3/2} (H_tf(\mathbf x)-f(\mathbf x)) dt.
\]
Using \eqref{eq:small t} and \eqref{eq:large t}, we obtain
\[
\sup _{\mathbf x\in \bH}(1+|\mathbf x|)^q\abs{X^K  \sqrt {-L}f(\mathbf x)} \lesssim  \int_0^{1} t^{-1/2} dt+\int_{1}^{\infty} t^{-3/2} dt<\infty. 
\]
\end{proof}

\begin{remark}
Consequently, $[W_j,\sqrt{-L}]$ is defined pointwisely on $\Scal(\bH^n)$.  Similarly, we can also show that  $2iT_jZ f\in \Scal(\bH^n)$ by using the subordination formula for the Poisson semigroup 
\[
P_tf=\frac{t}{2\sqrt{\pi}} \int_0^{\infty} e^{-\frac{t^2}{4s}} H_sf \frac{ds}{s^{3/2}}.
\]

\end{remark}

Consider the operator
\[
\Tcal_j=\int_0^{+\infty} y P_y ( W_j \sqrt{-L} +\sqrt{-L} W_j)P_y dy,
\]
we have the following equality.
\begin{proposition}\label{prop:equi}
For any $f\in \Scal(\bH^n)$, there holds
\begin{equation}\label{eq:commutator}
[W_j, \sqrt{-L}] f=2i\,\Tcal_jZf.
\end{equation}
\end{proposition}

In order to prove Proposition \ref{prop:equi}, we recall first the spectral decomposition of the sublaplacian on the Heisenberg group (see, for instance, \cite[Section 2]{RT16}).  Let $f\in L^2(\bH^n)$ be a radial function.  That is, for any $\mx=(x,y,z)\in \bH^n$, $f(\mx)=f(r,z)$ with $r=\|\mx\|$. Here $\|\mx\|$ denotes the Euclidean norm of the projection of $\mx$ onto the plane $\{z=0\}$. The spectral decomposition of the sublaplacian is given by 
\[
Lf(r,z)=-(2\pi)^{-n-1} \int_{-\infty}^{\infty} \sum_{k=0}^{\infty}(2k+n)|\ld| C_k^{\ld}(f^\ld) \varphi_k^{\ld}(r) e^{-i\ld z}|\ld|^n d\ld,
\]
where $ \varphi_k^{\ld}$ are   the scaled Laguerre functions 
\[
 \varphi_k^{\ld}(r)=L_k^{n-1}\Big(\frac{1}{2}|\ld|r^2\Big) e^{-\frac14|\ld| r^2},
\]
and $C_k^{\ld}(f^\ld)$ are the  Laguerre coefficients of the radial function $f^{\ld}$  given by 
\[
C_k^{\ld}(f^\ld)=C_{n,\ld}\frac{k!(n-1)!}{(k+n-1)!} \int_{\R^2} f^{\ld}(r) \varphi_k^{\ld}(r) dxdy.
\]
Notice that $\{ \varphi_k^{\ld}\}_{k=0}^{\infty}$ forms an orthogonal basis for the subspace consisting of radial functions in $L^2(\R^{2n})$.
The spectral decomposition of the associated heat semigroup is
\[
e^{tL}f(r,z)=(2\pi)^{-n-1} \int_{-\infty}^{\infty} \sum_{k=0}^{\infty}e^{-(2k+n)|\ld|t} C_k^{\ld}(f^\ld) \varphi_k^{\ld}(r) e^{-i\ld z}|\ld|^n d\ld.
\]
We also have the spectral decomposition for the heat kernel
\begin{equation}\label{eq:hk}
h_t(r,z)=C \int_{-\infty}^{\infty} \sum_{k=0}^{\infty}e^{-(2k+n)|\ld|t} \varphi_k^{\ld}(r) e^{-i\ld z}|\ld|^n d\ld.
\end{equation}

\begin{lemma}\label{lem:pt} For any $r>0$, $z\in \R$ and $t>0$,
we have
\begin{align}
\label{eq:equiv WL}
[W_j,\sqrt{-L}] h_t(r,z)=\int_{-\infty}^{\infty} \sum_{k=0}^{\infty}e^{-(2k+n)|\ld|t} [W_j,\sqrt{-L}] \varphi_k^{\ld}(r) e^{-i\ld z}|\ld|^n d\ld,
\end{align}
where the right hand side converges uniformly in $r$ and $z$.
\end{lemma}

\begin{proof}
For any $\mx=(x,y,z)\in \bH$, denote $\wtvp_k^{\ld}(\mx)=\wtvp_k^{\ld}(r,z)= \varphi_k^{\ld}(r) e^{-i\ld z}|\ld| $, where $r=\|\mx\|$. Then for any $t_0>0$
\[
H_{t_0}\wtvp_k^{\ld}=e^{-(2k+n)|\ld|t_0} \wtvp_k^{\ld},
\]
and 
\[
W_j\wtvp_k^{\ld}(\mx)=e^{(2k+n)|\ld|t_0} W_j H_{t_0} \wtvp_k^{\ld}(\mx)
=e^{(2k+n)|\ld|t_0} \int_{\bH^n} W_{j;\mx} h_{t_0}(\mx,\my) \wtvp_k^{\ld}(\my) d\my.
\]
Consequently,
\begin{align*}
\abs{W_j\sqrt{-L}\,\wtvp_k^{\ld}(\mx)}
&=\sqrt{(2k+n)|\ld| } \,e^{(2k+n)|\ld|t_0} \abs{ \int_{\bH^n} W_{j;\mx} h_{t_0}(\mx,\my) \wtvp_k^{\ld}(\my) d\my}
\\ &\le
 \sqrt{(2k+n)|\ld|}\,  e^{(2k+n)|\ld|t_0} \N{W_{j;\mx} h_{t_0}(\mx,\cdot)}_{L^1} \N{\wtvp_k^{\ld}}_{L^{\infty}}
 \\ &\le
Ct_0^{-1/2} \sqrt{(2k+n)|\ld|}\,  e^{(2k+n)|\ld|t_0}  \N{\wtvp_k^{\ld}}_{L^{\infty}}.
\end{align*}
Take $t_0=\frac{t}{2}$, then 
\begin{align*}
\abs{e^{-(2k+n)|\ld|t} W_j\sqrt{-L} \varphi_k^{\ld}(r) e^{-i\ld z}|\ld|^n} 
\le 
Ct^{-1/2}   |\ld|^n \sqrt{(2k+n)|\ld|} e^{-(2k+n)|\ld|t/2}.
\end{align*}
This implies that  the right hand side of \eqref{eq:equiv WL} converges if $[W_j,\sqrt{-L}]$ is replaced by $W_j\sqrt{-L}$. Hence
we have
\[
W_j\sqrt{-L} h_t(r,z)=\int_{-\infty}^{\infty} \sum_{k=0}^{\infty}e^{-(2k+n)|\ld|t} W_j\sqrt{-L}\varphi_k^{\ld}(r) e^{-i\ld z}|\ld|^n d\ld.
\]

In order to show \eqref{eq:equiv WL}, it remains to consider $\sqrt{-L}W_j$. Using right invariant operators and integration by parts, we have
\begin{align*}
\abs{\sqrt{-L}W_j \,\wtvp_k^{\ld}(\mx)}
&=  
\abs{\sqrt{-L}W_j H_{t_0} \wtvp_k^{\ld}(\mx)}
=
\abs{ H_{t_0}  \sqrt{-\hat L}\hat W_j\,\wtvp_k^{\ld}(\mx) }\\
&= 
\abs{ \int_{\bH^n} h_{t_0}(\mx,\my)  \sqrt{-\hat L}\hat W_j\,\wtvp_k^{\ld}(\my) d\my}
\\ &=
\abs{ \int_{\bH^n}\hat W_j\sqrt{-\hat L}\,h_{t_0}(\mx,\my)  \,\wtvp_k^{\ld}(\my) d\my}\\
&\le 
 \int_{\bH^n} \abs{\hat W_j\sqrt{-\hat L}\,h_{t_0}(\mx,\my) }d\my \,\N{\wtvp_k^{\ld}}_{L^{\infty}}.
\end{align*}
 By Lemma \ref{lem:Schwartz}, we see that the integral $ \int_{\bH^n} \abs{\hat W_j\sqrt{-\hat L}\,h_{t_0}(\mx,\my) }d\my $ converges. This leads to \eqref{eq:equiv WL} with $[W_j,\sqrt{-L}]$ being replaced by $W_j\sqrt{-L}$ and hence \eqref{eq:equiv WL}.
\end{proof}
\begin{remark} In the same way, we can also prove that 
\begin{align}
\label{eq:equiv TZ}
2i\,\Tcal_jZ h_t(r,z)=\int_{-\infty}^{\infty} \sum_{k=0}^{\infty}e^{-(2k+n)|\ld|t} 2i\,\Tcal_jZ \varphi_k^{\ld}(r) e^{-i\ld z}|\ld|^n d\ld.
\end{align}
\end{remark}

\begin{proof}[Proof of Proposition \ref{prop:equi}]
 We first show that 
\begin{equation}\label{eq:pt}
[W_j, \sqrt{-L}]\, h_t(r,z)=2i\,\Tcal_jZ\,h_t(r,z).
\end{equation}
By Lemma \ref{lem:pt}, it suffices to show that,   for any $\ld\in \R$ and any $k \in \bN$,
\[
[W_j,\sqrt{-L}] \wtvp_k^{\lambda}=2i\,\Tcal_jZ\,\wtvp_k^{\lambda}.
\]
Before computation, we recall that $Z$ commutes with $L$ and $W_j$. Then $LZ\,\wtvp_k^{\lambda}=ZL\,\wtvp_k^{\lambda}=-(2k+n)|\ld|$, i.e., $Z\,\wtvp_k^{\lambda}$ is an eigenfunction of the eigenvalue $-(2k+n)|\ld|$ for $L$.  Thus we have
\begin{align*}
2i\,\Tcal_jZ\,\wtvp_k^{\lambda} &=2i\int_0^{+\infty} y P_y \brak{ W_j \sqrt{-L} +\sqrt{-L} W_j}P_y Z\,\wtvp_k^{\lambda}\,dy
\\&=
2i\int_0^{+\infty} y e^{-y\sqrt{-L}} \brak{W_j \sqrt{(2k+n)|\ld|} +\sqrt{-L} W_j}e^{-y\sqrt{(2k+n)|\ld|}}Z\,\wtvp_k^{\lambda}\, dy
\\&=
2i\int_0^{+\infty} y e^{-y\brak{\sqrt{-L}+\sqrt{(2k+n)|\ld|}}}  \brak{\sqrt{(2k+n)|\ld|} +\sqrt{-L}} W_jZ\,\wtvp_k^{\lambda}\,dy
\\&=
\int_0^{+\infty} y e^{-y\brak{\sqrt{-L}+\sqrt{(2k+n)|\ld|}}} \brak{\sqrt{(2k+n)|\ld|} +\sqrt{-L}} (LW_j-W_jL)\wtvp_k^{\lambda}\,dy
\\&=
\int_0^{+\infty} y e^{-y\brak{\sqrt{-L}+\sqrt{(2k+n)|\ld|}}} \brak{\sqrt{(2k+n)|\ld|} +\sqrt{-L}} (L+(2k+n)|\ld|)W_j\,\wtvp_k^{\lambda}\, dy
\\&=
\brak{\sqrt{(2k+n)|\ld|}-\sqrt{-L}} W_j\,\wtvp_k^{\lambda}=[W_j, \sqrt{-L}] \wtvp_k^{\lambda}. 
\end{align*}
Here, in the third equality, we use \eqref{eq:WL} and the fact that $Z$ commutes with $W_j$.

With this preparation, we can now prove \eqref{eq:commutator}.
By the left invariance, it is enough to prove the equality at $\mathbf x=0$. We notice that 
\begin{equation}\label{eq:commutator0}
[W_j,\sqrt{-L}] f(0)=\lim_{t\to 0} H_t [W_j,\sqrt{-L}] f(0). 
\end{equation}
Using integration by parts and also \eqref{eq:pt}, we have
\begin{align*}
H_t[ W_j,\sqrt{-L}]f(0)
&= \int_{\bH^n} h_t(\my) [W_j,\sqrt{- L}]f(\my) d\my\\
&=
- \int_{\bH^n} [W_j,\sqrt{- L}] h_t(\my) f(\my) d\my
\\ &=
- \int_{\bH^n} 2i\,\Tcal_j Z h_t(\my) f(\my) d\my\\
&= \int_{\bH^n} h_t(\my) 2i \,\Tcal_j Z  f(\my) d\my= H_t  2i \,\Tcal_j Zf(0),
\end{align*}
where the second-to-last equality  holds since $Z$ commutes with $L$. This gives us 
\[
\lim_{t\to 0} H_t [W_j,\sqrt{-L}] f(0)=2i \Tcal_j Zf(0),
\]
which leads to \eqref{eq:commutator0} and hence \eqref{eq:commutator}.

\end{proof}

Finally, we conclude  that
\begin{proposition}\label{Hei}
Let  $1\le j\le n$ and $f\in \Scal(\bH^n)$. Then we have
\[
\| \text{ } [ W_j, \sqrt{-L} ] f \text{  }  \|_p \le \sqrt{2}(p^*-1)\| Z f \|_p.
\]
\end{proposition}
\begin{proof}
By Proposition \ref{prop:equi}, we have 
\[
\| \text{ } [ W_j, \sqrt{-L} ] f \text{  }  \|_p \le 2\| \Tcal_jZ f \|_p.
\]
It suffices to work on $\Tcal_j$. Following the proof of Theorem \ref{thm:MT}, we have Gundy-Varopoulos type probabilistic representation of $\Tcal_j$ as follows
\[
\Tcal_j f(x)=- \frac{1}{2} \lim_{y_0\to \infty} \bE_{y_0} \brak{ \int_0^{\tau} \Acal_j(\nabla,Z)^{\mathrm T} Q f(\Y_s, B_s) (d\beta_s,dB_s)\mid \Y_{\tau}=x},
\] 
where $\nabla=(X_1,\cdots,X_n, Y_1,\cdots, Y_n)$, $ (\Y_t)_{t\ge 0}$ is the diffusion process on $\bH^n$ with generator  $L$, $\beta_s$ is the Brownian motion on $\R^{2n}$, and $\Acal_j$ is a $(2n+1)\times (2n+1)$ matrix  as follows:
\[
a_{j(2n+1)}=1,\, a_{(n+j)(2n+1)}=i,\,a_{(2n+1)j}=-1,\, a_{(2n+1)(n+j)}=-i; \text { and otherwise }0.
\]
Notice that $\N{\Acal_j}=\sqrt2$, therefore by Theorem \ref{thm:MT2},
\[
\| \text{ } [ W_j, \sqrt{-L} ] f \text{  }  \|_p \le 2\| \Tcal_jZ f \|_p \le 2\N{\Acal_j} \frac{(p^*-1)}{2} \| Z f \|_p =  \sqrt2(p^*-1)\| Z f \|_p. 
\]
\end{proof}

\subsection{Example 3. $\SU(2)$}

Consider the Lie group $\SU(2)$, i.e., the group of $2\times2 $ complex unitary matrices of determinant $1$. Its Lie algebra $\su(2)$ consists of $2\times2 $ complex skew adjoint matrices of trace $0$. A basis of $\su(2)$ is formed by the Pauli matrices
\[
X= \begin{pmatrix}
    0 & 1 \\
    -1 & 0 
\end{pmatrix},
\quad 
Y= \begin{pmatrix}
    0 & i \\
    i &  0 
\end{pmatrix},
\quad
Z= \begin{pmatrix}
    i & 0 \\
    0 & -i 
\end{pmatrix},
\]
for which the commutation relations hold
\[
[X,Y]=2Z, \quad [Y,Z]=2X, \quad [Z,X]=2Y.
\]
Denote by $X,Y,Z$ the left invariant vector fields on $\SU(2)$ corresponding to the Pauli matrices. We shall be interested in the operator
\[
L=X^2+Y^2.
\]

Let $(H_t)_{t>0}=(e^{t L})_{t>0}$  be the heat semigroup generated by $L$ and $h_t(\mx)$ be the corresponding heat kernel at $0$.  Let  $(P_t)_{t>0}=(e^{-t\sqrt{-L}})_{t>0}$ be the Poisson semigroup and $p_t$ be the Poisson kernel at $0$. 
We use the cylindric coordinates introduced in \cite{CS01} 
\[
(r,\theta,z) \to \exp(r\cos \theta X+r\sin \theta Y) \exp(zZ)=\begin{pmatrix}
   \cos(r) e^{iz} & \sin(r) e^{i(\theta-z)} \\
    -\sin(r) e^{-i(\theta-z)} & \cos(r) e^{iz}
\end{pmatrix}
\]
with 
\[
0\le r\le \frac{\pi}2,\quad \theta\in [0,2\pi], \quad z\in [-\pi,\pi].
\]
The heat kernel at $0$ depends only on $r$ and $z$, which we denote by $h_t(r,z)$. The spectral decomposition of $h_t(r,z)$ can be found in \cite{BauB09}: for $t>0$, $0\le r<\frac{\pi}2$, $z\in [-\pi,\pi]$,
\[
h_t(r,z)=\sum_{n={-\infty}}^{+\infty} \sum_{k=0}^{+\infty} (2k+|n|+1) e^{-(4k(k+|n|+1)+2|n|)t} e^{inz} (\cos r)^{|n|} P_k^{0,|n|}(\cos 2r).
\]

Let $W=X+iY$ be the complex gradient. Then the Lie algebra structure gives us
\begin{equation}\label{eq:WL(SU)}
WL=(L-4iZ+4)W.
\end{equation}
Consider the operator
\[
\Tcal=\int_0^{+\infty} y P_y ( W \sqrt{-L} +\sqrt{-L} W)P_y dy.
\]

\begin{proposition} \label{prop:equiSU2}
For any smooth function $f$, there holds
\begin{equation}\label{eq:commutatorSU}
[W, \sqrt{-L}] f=\Tcal(4iZ+4)f.
\end{equation}
\end{proposition}
\begin{proof}
Denote $\Phi_{n,k}(r,z)=e^{inz} (\cos r)^{|n|} P_k^{0,|n|}(\cos 2r)$. Then $\Phi_{n,k}$ is a eigenfunction of $L$ corresponding to the eigenvalue $-\ld_{n,k}=-4k(k+|n|+1)-2|n|$.

Similarly as on Heisenberg groups (details are neglected), it suffices to check \eqref{eq:commutatorSU} acting on eigenfunctions $\Phi_{n,k}(r,z)$. 
We first compute $\Tcal\Phi_{n,k}(r,z)$, 
\begin{align*}
\Tcal \Phi_{n,k} &=\int_0^{+\infty} y P_y ( W \sqrt{-L} +\sqrt{-L} W)P_y \Phi_{n,k} dy
\\&=
\int_0^{+\infty} y e^{-y\sqrt{-L}} ( W \sqrt{\ld_{n,k}} +\sqrt{-L} W)e^{-y\sqrt{\ld_{n,k}}}\Phi_{n,k} dy
\\&=
\int_0^{+\infty} y e^{-y(\sqrt{-L}+\sqrt{\ld_{n,k}})} (\sqrt{\ld_{n,k}} +\sqrt{-L}) W \Phi_{n,k} dy
\\&=
(\sqrt{-L}+\sqrt{\ld_{n,k}})^{-1} W\Phi_{n,k}. 
\end{align*}
Observe that $[L,Z]=0$, then $LZ\Phi_{n,k}=ZL\Phi_{n,k}=-\ld_{n,k} Z\Phi_{n,k}$, i.e., $Z\Phi_{n,k}$ is also an eigenfunction of eigenvalue $-\ld_{n,k}$. Hence the above computation also works for $Z\Phi_{n,k}$ and we obtain that
\[
\Tcal(4iZ+4)\Phi_{n,k}=(\sqrt{-L}+\sqrt{\ld_{n,k}})^{-1} W(4iZ+4)\Phi_{n,k}.
\]
Observe also $[W,Z]=2iW$, which leads to $(4iZ-4)W=W(4iZ+4)$. Consequently
\begin{equation}\label{eq:TPhi}
\Tcal(4iZ+4)\Phi_{n,k}=(\sqrt{-L}+\sqrt{\ld_{n,k}})^{-1} (4iZ-4)W\Phi_{n,k}.
\end{equation}

Next compute the commutator acting on $\Phi_{n,k}$, we have
\[
[W, \sqrt{-L}] \Phi_{n,k}=W\sqrt{\ld_{n,k}}\Phi_{n,k}-\sqrt{-L}W\Phi_{n,k}=(\sqrt{\ld_{n,k}}-\sqrt{-L}) W\Phi_{n,k}.
\]
In addition, 
\begin{align*}
(\sqrt{-L}+\sqrt{\ld_{n,k}})(\sqrt{\ld_{n,k}}-\sqrt{-L}) W\Phi_{n,k}&=(\ld_{n,k}+L)W\Phi_{n,k}\\
&=LW\Phi_{n,k}-WL\Phi_{n,k}\\
&=(4iZ-4) W\Phi_{n,k},
\end{align*}
where the last equality is due to \eqref{eq:WL(SU)}.

Summarize the above three equalities we conclude the proof for \eqref{eq:commutatorSU}.
\end{proof}

As a conclusion, we have
\begin{proposition}\label{SU2}
Let $1<p<\infty$. Then for any smooth function $f$,
\[
\| \text{ } [ W, \sqrt{-L} ] f \text{  }  \|_p \le  2\sqrt{2}(p^*-1)  \| (iZ+1) f \|_p.
\]
\end{proposition}

\begin{proof}
By Proposition \ref{prop:equiSU2}, we have 
\[
\| \text{ } [ W, \sqrt{-L} ] f \text{  }  \|_p \le \| \Tcal (4iZ+4) f \|_p.
\]
 Following the proof of Theorem \ref{thm:MT}, we have Gundy-Varopoulos type probabilistic representation of $\Tcal$ as follows
\[
\Tcal_j f(x)=- \frac{1}{2} \lim_{y_0\to \infty} \bE_{y_0} \brak{ \int_0^{\tau} \Acal_j(\nabla,Z)^{\mathrm T} Q f(\Y_s, B_s) (d\beta_s,dB_s)\mid \Y_{\tau}=x},
\] 
where $\nabla=(X, Y)$, $ (\Y_t)_{t\ge 0}$ is the diffusion process on $\SU(2)$ with generator  $L$, $\beta_s$ is the Brownian motion on $\R^{2}$, and $\Acal$ is a $3\times 3$ matrix  as follows:
\[
a_{13}=1,\, a_{23}=i,\,a_{31}=-1,\, a_{32}=-i; \text { and otherwise }0.
\]
Notice that $\N{\Acal}=\sqrt2$, therefore by Theorem \ref{thm:MT2},
\begin{align*}
\| \text{ } [ W, \sqrt{-L} ] f \text{  }  \|_p &\le \| \Tcal (4iZ+4) f \|_p \\
&\le \N{\Acal} \frac{(p^*-1)}{2} \| (4iZ+4) f \|_p\\
&=   2\sqrt2(p^*-1)  \| (iZ+1) f \|_p. 
\end{align*}
\end{proof}

\section{Riesz transforms on vector bundles}

Our general results are easily generalized in the framework of vector bundles. This framework is more adapted to the study of Riesz transforms vectors.  

\subsection{A general theorem}

Let $\mathbb{M}$ be a $d$-dimensional  smooth complete Riemannian
manifold and let $\mathcal{E}$ be a finite-dimensional vector
bundle over $\mathbb{M}$. We denote by $\Gamma ( \mathbb{M},
\mathcal{E} )$ the space of smooth  sections of this bundle. Let now $\nabla$
denote a metric connection on $\mathcal{E}$. We consider an operator on $\Gamma ( \mathbb{M},
\mathcal{E} )$  that can be written as
\[
\mathcal{L}=\mathcal{F}+\nabla_{0}+ \sum_{i=1}^d \nabla_{i}^2,
\]
where
\[
\nabla_i=\nabla_{X_i}, \quad 0 \le i \le d,
\]
and the $X_i$'s  are smooth vector fields on $\mathbb{M}$ and 
$\mathcal{F}$  is a smooth symmetric and non positive potential (that is a smooth section of the
bundle $\mathbf{End}(\mathcal{E})$). We will assume that $\mathcal{L}$ is non-positive and essentially self-adjoint on the space $\Gamma_0(\M,\mathcal{E})$ of smooth and compactly supported sections. We consider then a first order differential operator $d_a$ on $\Gamma ( \mathbb{M},
\mathcal{E} )$ that can be written as
\[
d_a=\sum_{i=1}^d a_i \nabla_{X_i},
\]
where $a_1,\cdots,a_d$ are smooth sections of the bundle $\mathbf{End}(\mathcal{E})$. Our main assumptions  are that 
\[
d_a \mathcal{L} \eta =\mathcal{L} d_a \eta,\quad \eta \in \Gamma ( \mathbb{M},
\mathcal{E} ),
\]
and that 
\[
\| d_a \eta \|^2 \le C \sum_{i=1}^d \| \nabla_{X_i} \eta \|^2, \quad \eta \in \Gamma ( \mathbb{M},
\mathcal{E} ),
\]
for some constant $C\ge 0$. Several instances of such situations will be illustrated in the sequel. Our main theorem is the following:

\begin{theorem} \label{JKNM}
For $1<p<\infty$,
\[
\| d_a (-\mathcal{L})^{-1/2} \eta \|_p \le 6C(p^*-1) \| \eta \|_p.
\]
\end{theorem}

The proof follows the same lines as in the previous section.
Let $(Y_t)_{t\ge 0}$ be the diffusion process on $\M$ with generator  $\sum_{i=1}^d X_i^2+X_0$ started from the distribution $\mu$. We assume that $(Y_t)_{t\ge 0}$ is non explosive, thus as before  it can then be constructed via the Stratonovitch stochastic differential equation
\[
dY_t=X_0(t)dt+\sum_{i=1}^d X_i(Y_t)\circ d\beta_t^i, 
\]
where $\beta_t=(\beta_t^1,\cdots,\beta_t^d)$ is the Brownian motion on $\R^d$ with generator $\sum_{i=1}^d \frac{\partial^2}{\partial x_i^2}$.

Let $(B_t)_{t\ge 0}$ be a one-dimensional Brownian motion on $\R$, with generator $\frac{d^2}{dy^2}$ starting from $y_0>0$, which is independent from $(Y_t)_{t\ge 0}$. Then $\bE(B_t^2)=2t$. Set the stopping time
$$
\tau=\inf\{t>0, \,B_t=0\}.
$$

We denote by $H_t=e^{tL}$ the heat semigroup and by $P_y=e^{-y \sqrt{-L}}$ the Poisson semigroup. In that framework, there is a well known Feynman-Kac representation for the semigroup $H_t$.
 
More precisely,  consider the stochastic parallel transport along $Y_t$, $\theta_t: \mathcal{E}_{Y_t} \to \mathcal{E}_{Y_0}$ and the multiplicative functional $(\mathcal{M} _t)_{t \ge 0}$, solution of the equation
 \[
 \frac{d \mathcal{M}_t}{dt}= \mathcal{M}_t \theta_t  \mathcal{F}  \theta_t^{-1}, \quad \mathcal{M}_0=\mathbf{Id}.
 \]
By the Feynman-Kac formula, the heat semigroup $H_t$ acting on $\eta \in \Gamma_0^\infty(\M,\mathcal{E})$ can then be written as 
\[
H_t\eta(x)=\bE_x\brak{ \mathcal{M}_t \theta_t \eta (Y_t)}.
\]
For $\eta \in \Gamma_0^\infty(\M,\mathcal{E})$, denote 
\[
Q\eta(x,y)=P_y \eta(x)=e^{-y\sqrt{-\mathcal{L}}} \eta(x),
\]
and 
\[
M_s^\eta=  \mathcal{M}_{s\wedge \tau}  \theta_{s \wedge \tau} Q\eta (Y_{s\wedge \tau},B_{s\wedge \tau}).
\]
As in Lemma \ref{martingale:lk}, we can easily prove that $M^\eta$ is a martingale.
We have then the following Gundy-Varopoulos type representation :

\begin{lemma}\label{lem:CE2}
Let $\eta \in \Gamma_0^\infty(\M,\mathcal{E})$. For almost all $x\in \M$, we have
\[
d_a  (-\mathcal{L})^{-1/2} \eta (x)=-2 \lim_{y_0\to \infty} \bE_{y_0} \brak{ \theta_\tau^{-1} \mathcal{M}_\tau^*\int_0^{\tau}  (\mathcal{M}_s^*)^{-1} \theta_s   d_a Q f(Y_s, B_s) dB_s\mid Y_{\tau}=x}.
\]
\end{lemma}

\begin{proof}
Let $\alpha\in\Gamma_0^\infty(\M,\mathcal{E})$ and $\alpha\in\Gamma_0^\infty(\M,\mathcal{E})$. Note that 
\[
M_{\tau}^\eta=\mathcal{M}_{ \tau}  \theta_{ \tau}\alpha(Y_{\tau}).
\]
By  It\^o isometry, we have
\begin{align*}
&\int_{\M} \left\langle  \alpha(x), \mathbb{E}\brak{ \theta_\tau^{-1} \mathcal{M}_\tau^*\int_0^{\tau}   (\mathcal{M}_s^*)^{-1} \theta_s  d_a Q \eta(Y_s, B_s) dB_s\mid Y_{\tau}=x} \right\rangle d\mu(x)
\\&=
\bE_{y_0} \brak{ \left\langle \mathcal{M}_{ \tau}  \theta_{ \tau}\alpha(Y_{\tau}) , \int_0^{\tau}  (\mathcal{M}_s^*)^{-1} \theta_s  d_a Q \eta(Y_s, B_s) dB_s \right\rangle}
\\&=
2 \bE_{y_0} \brak{\int_0^{\tau} \langle \partial_y Q\alpha(Y_s,B_s) , d_aQ \eta(Y_s, B_s) \rangle ds}
\\ &=
2\int_{\M} \int_0^{\infty} (y_0 \wedge y) \langle \partial_y Q\alpha(x,y) ,d_aQ \eta(x,y)\rangle dy d\mu(x).
\end{align*}
The last equality follows from the facts that the Green function of the Brownian motion is killed at $0$ and $Y_s$ is distributed according to $\mu$. 
Using finally the commutation between $d_a$ and $\mathcal{L}$ one deduces
\[
2\int_{\M} \int_0^{\infty}y \langle \partial_y Q\alpha(x,y) ,d_aQ \eta(x,y)\rangle dy d\mu(x)=-\frac{1}{2} \int_\M  \left\langle  \alpha(x), d_a  (-\mathcal{L})^{-1/2} \eta (x) \right\rangle d\mu(x).
\]
\end{proof}

The proof of Theorem \ref{JKNM} now follows the lines of the proof of Theorem \ref{thm:MT}.

\subsection{Example 1. The Riesz transform on  forms}

Let $\mathbb{M}$ be a $d$-dimensional  smooth, oriented, complete and stochastically complete Riemannian
manifold. We first briefly recall some basic facts on
Fermion calculus on the exterior algebra of a finite
dimensional vector space, as can be found in Section 2.2.2 of
\cite{Ros}. 
 Let $V$ be a
$d$-dimensional Euclidean vector space. We denote $V^\ast$ its
dual and
\[
\wedge V^\ast=\bigoplus_{k \ge 0} \wedge^k V^\ast,
\]
the exterior algebra. If $u \in V^\ast$, we denote $a^\ast_u$ the
map $\wedge V^\ast \rightarrow \wedge V^\ast$, such that $a^\ast_u
(\omega)=u \wedge \omega$. The dual map is denoted $a_u$. Let now
$\theta_1$, ..., $\theta_d$ be an orthonormal basis of $V^\ast$.
We denote $a_i=a_{\theta_i}$. We have the basic rules of Fermion
calculus
\[
\{ a_i ,a_j \}=0, \{ a^\ast_i ,a^\ast_j \}=0, \{ a_i ,a^\ast_j
\}=\delta_{ij},
\]
where $\{ \cdot , \cdot \}$ stands for the anti-commutator: $\{
a_i ,a_j \}=a_i a_j +a_j a_i$. If $I$ and $J$ are two words with
$1 \le i_1 < \cdots < i_k \le d$ and $1 \le j_1 < \cdots < j_l \le
d$, we denote
\[
A_{IJ}=a^\ast_{i_1} \cdots a^\ast_{i_k}   a_{j_1} \cdots a_{j_l}.
\]
The family of all the possible $A_{IJ}$ forms a basis of the
$2^{2d}$-dimensional vector space $\mathbf{End} \left( \wedge
V^\ast\right)$.   We can carry the Fermionic construction on the tangent spaces of
the manifold $\mathbb{M}$. Let $e_i$ be a local orthonormal frame
and let $\theta_i$ be its dual frame. In that frame, we can express the exterior derivative as
\begin{align}\label{exterior}
d=\sum_i a_i^* \nabla_{e_i} .
\end{align}
Let us therefore observe that $\| d \eta \|^2 \le \sum_{i=1}^d \| \nabla_{e_i} \eta \|^2$.
The curvature endomorphism (Weitzenb\"ock curvature)
is then defined by
\[
\mathcal{F}=-\sum_{ijkl} R_{ijkl}a_i^\ast a_j a_k^\ast a_l
\]
where
\[
R_{ijkl}=\left\langle R(e_j,e_k)e_l,e_i \right\rangle,
\]
with $R$ Riemannian curvature of $\mathbb{M}$. The celebrated Weitzenb\"ock formula writes
\[
\mathcal{L}=\Delta - \mathcal{F},
\]
where $\mathcal{L} =-d d^\ast -d^\ast d$ is the Hodge-DeRham Laplacian
and $\Delta$ the Bochner Laplacian. Let us recall that if $e_i$ is
a local orthonormal frame, we have the following explicit formula
for $\Delta$:
\begin{align}\label{laplacian}
\Delta=\sum_{i=1}^d (\nabla_{e_i}\nabla_{e_i}- \nabla_{
\nabla_{e_i} e_i }),
\end{align}
where $\nabla$ is the Levi-Civita connection. The following theorem is then an application of Theorem \ref{JKNM}.

\begin{theorem}
Assume $\mathcal{F} \ge 0$, then
\[
\| d (-\mathcal{L})^{-1/2} \eta \|_p \le 6(p^*-1) \| \eta \|_p
\]
\end{theorem}

\begin{remark}
Let us observe that the expressions \ref{exterior} and \ref{laplacian} are only defined locally in a given frame, however Lemma \ref{lem:CE2} is coordinate free and therefore holds in the present setting (see Theorem 3.2 in \cite{LiXD08}).
\end{remark}

\subsection{Example 2. The Riesz transform on spinors}

We first review some basic constructions in spin geometry.  Let $V$ be an oriented $d$ dimensional Euclidean space. We assume
that the dimension $d$ is even. The Clifford algebra
$\mathbf{Cl}(V)$ over $V$ is the algebra
\[
\mathbf{T} (V)=\mathbb{R} \oplus V \oplus (V \otimes V) \oplus
\cdots
\]
quotient by the relations
\begin{align}\label{relation-clifford}
u \otimes v +v \otimes u +2 \langle u,v \rangle 1 =0.
\end{align}
Let $e_1,...,e_d$ be an oriented basis of $V$. The family
\[
e_{i_1} ... e_{i_k}, \quad 0 \le k \le d, \quad 1 \le i_1 <...<i_k
\le d,
\]
forms a basis of the vector space $\mathbf{Cl}(V)$ which  is therefore of dimension
$2^d$. In $\mathbf{T} (V)$ we can distinguish elements that are
even from elements that are odd. This leads to a decomposition:
\[
\mathbf{Cl}(V)=\mathbf{Cl}^-(V) \oplus \mathbf{Cl}^+(V),
\]
with $V \subset \mathbf{Cl}^-(V)$. A Clifford module is a vector space $E$ over $\mathbb{R}$ (or
$\mathbb{C}$) that is also a $\mathbf{Cl}(V)$-module and that
admits a direct sum decomposition
\[
E=E^- \oplus E^+
\]
with
\[
\mathbf{Cl}^-(V) \cdot E^- \subset E^-, \quad \mathbf{Cl}^+(V)
\cdot E^+ \subset E^+.
\]
It can be shown that there is a unique Clifford module $S$, called
the spinor module over $V$ such that:
\[
\mathbf{End} (S) \simeq \mathbb{C} \otimes \mathbf{Cl}(V).
\]
In particular $\dim S=2^{\frac{d}{2}}$. 
If $\psi \in \mathfrak{so} (V)$, that is if $\psi : V \rightarrow
V$ is a skew-symmetric map, we define
\[
D\psi=\frac{1}{2}\sum_{1 \le i<j \le d} \langle \psi (e_i), e_j
\rangle e_i e_j \in \mathbf{Cl}(V),
\]
and observe that $D[\psi_1,\psi_2]=[D\psi_1,D\psi_2]$. The set
$\mathbf{Cl}^2(V)=D \mathfrak{so} (V)$ is therefore a Lie algebra.
The Lie group $\mathbf{Spin} (V)$ is the group obtained by
exponentiating $\mathbf{Cl}^2(V)$ inside the Clifford algebra
$\mathbf{Cl}(V)$; It is the two-fold universal covering of the
orthogonal group $\mathbf{SO}(V)$. It can also be described as the
set of $a \in \mathbf{Cl}(V)$ such that:
\[
a=v_1 ... v_{2k}, \quad 1 \le k \le \frac{d}{2},\quad v_i \in V,
\quad \parallel v_i \parallel=1.
\]

We now come back to the manifold setting and carry the above
constructions on the cotangent spaces of a spin manifold. So, let $\mathbb{M}$ be a  $d$-dimensional, oriented, complete and stochastically complete  Riemannian manifold. We assume that $d$ is even. We furthermore
assume that $\mathbb{M}$ admits a spin structure: That is, there
exists a principal bundle on $\mathbb{M}$ with structure group
$\mathbf{Spin}(\mathbb{R}^d)$ such that the bundle charts are
compatible with the universal covering
$\mathbf{Spin}(\mathbb{R}^d) \rightarrow
\mathbf{SO}(\mathbb{R}^d)$. This bundle will be denoted
$\mathcal{SP} (\mathbb{M})$ and $\pi$ will denote the canonical
surjection. The spin bundle $\mathcal{S}$ over $\mathbb{M}$ is the
vector bundle such that for every $x \in \mathbb{M}$,
$\mathcal{S}_x$ is the spinor module over the cotangent space
$\mathbf{T}^*_x \mathbb{M}$. At each point $x$, there is therefore
a natural action of $\mathbf{Cl} (\mathbf{T}^*_x \mathbb{M})\simeq
\mathbf{End} (\mathcal{S}_x)$; this action will be denoted by
$\mathbf{c}$. On $\mathcal{S}$, there is a canonical elliptic first-order
differential operator called the Dirac operator and denoted
$\mathbf{D}$. In a local orthonormal frame $e_i$, with dual frame
$e_i^*$,  the Dirac operator is given by
\[
\mathbf{D}=\sum_{i=1}^d c(e^*_i) \nabla_{e_i},
\]
where $\nabla$ is the Levi-Civita connection. As a consequence, $\| \mathbf{D} \eta \|^2 \le \sum_{i=1}^d \| \nabla_{e_i} \eta \|^2$.  We also have an analogue
of Weitzenb\"ock formula which is the celebrated Lichnerowicz
formula (see Theorem 3.52 in \cite{Ber-Ge-Ve}):
\[
-\mathbf{D}^2=\Delta -\frac{s}{4},
\]
where $s$ is the scalar curvature of $\mathbb{M}$ and $\Delta$ is
given in a local orthonormal frame $e_i$ by
\[
\Delta=\sum_{i=1}^d (\nabla_{e_i}\nabla_{e_i}- \nabla_{
\nabla_{e_i} e_i }).
\]

The following theorem is then an application of Theorem \ref{JKNM}.

\begin{theorem}
Assume that the scalar curvature $s \ge 0$, then
\[
\| \mathbf{D} (-\mathbf{D}^2)^{-1/2} \eta \|_p \le 6(p^*-1) \| \eta \|_p.
\]
\end{theorem}
                         
\bibliographystyle{plain}
\bibliography{bibliography.bib}

\begin{thebibliography}{10}

\bibitem{Ale92}
G.~Alexopoulos.
\newblock An application of homogenization theory to harmonic analysis:
  {H}arnack inequalities and {R}iesz transforms on {L}ie groups of polynomial
  growth.
\newblock {\em Canad. J. Math.}, 44(4):691--727, 1992.

\bibitem{MR1650585}
N.~Arcozzi.
\newblock Riesz transforms on compact {L}ie groups, spheres and {G}auss space.
\newblock {\em Ark. Mat.}, 36(2):201--231, 1998.

\bibitem{ACDH04}
P.~Auscher, T.~Coulhon, X.~T. Duong, and S.~Hofmann.
\newblock Riesz transform on manifolds and heat kernel regularity.
\newblock {\em Ann. Sci. {\'E}cole Norm. Sup. (4)}, 37(6):911--957, 2004.

\bibitem{AMR08}
P.~Auscher, A.~McIntosh, and E.~Russ.
\newblock Hardy spaces of differential forms on {R}iemannian manifolds.
\newblock {\em J. Geom. Anal.}, 18(1):192--248, 2008.

\bibitem{BB13}
R.~Ba\~nuelos and F.~Baudoin.
\newblock Martingale transforms and their projection operators on manifolds.
\newblock {\em Potential Anal.}, 38(4):1071--1089, 2013.

\bibitem{BO15}
R.~Ba\~nuelos and A.~Os{\c e}kowski.
\newblock Sharp martingale inequalities and applications to {R}iesz transforms
  on manifolds, {L}ie groups and {G}auss space.
\newblock {\em J. Funct. Anal.}, 269(6):1652--1713, 2015.

\bibitem{BW95}
R.~Ba\~nuelos and G.~Wang.
\newblock Sharp inequalities for martingales with applications to the
  {B}eurling-{A}hlfors and {R}iesz transforms.
\newblock {\em Duke Math. J.}, 80(3):575--600, 1995.

\bibitem{Bak87}
D.~Bakry.
\newblock {\'E}tude des transformations de {R}iesz dans les vari{\'e}t{\'e}s
  riemanniennes {\`a} courbure de {R}icci minor{\'e}e.
\newblock In {\em S{\'e}minaire de {P}robabilit{\'e}s, {XXI}}, volume 1247 of
  {\em Lecture Notes in Math.}, pages 137--172. Springer, Berlin, 1987.

\bibitem{BBBC}
D.~Bakry, F.~Baudoin, M.~Bonnefont, and D.~Chafa\"\i.
\newblock On gradient bounds for the heat kernel on the {H}eisenberg group.
\newblock {\em J. Funct. Anal.}, 255(8):1905--1938, 2008.

\bibitem{BauB09}
F.~Baudoin and M.~Bonnefont.
\newblock The subelliptic heat kernel on {${\rm SU}(2)$}: representations,
  asymptotics and gradient bounds.
\newblock {\em Math. Z.}, 263(3):647--672, 2009.

\bibitem{BauB16}
F.~Baudoin and M.~Bonnefont.
\newblock Reverse {P}oincar{\'e} inequalities, isoperimetry, and {R}iesz
  transforms in {C}arnot groups.
\newblock {\em Nonlinear Anal.}, 131:48--59, 2016.

\bibitem{BG13}
F.~Baudoin and N.~Garofalo.
\newblock A note on the boundedness of {R}iesz transform for some subelliptic
  operators.
\newblock {\em Int. Math. Res. Not. IMRN}, (2):398--421, 2013.

\bibitem{Ber-Ge-Ve}
N.~Berline, E.~Getzler, and M.~Vergne.
\newblock {\em Heat kernels and {D}irac operators}, volume 298 of {\em
  Grundlehren der Mathematischen Wissenschaften [Fundamental Principles of
  Mathematical Sciences]}.
\newblock Springer-Verlag, Berlin, 1992.

\bibitem{CarbD13}
A.~Carbonaro and O.~Dragi{\v c}evi{\'c}.
\newblock Bellman function and linear dimension-free estimates in a theorem of
  {B}akry.
\newblock {\em J. Funct. Anal.}, 265(7):1085--1104, 2013.

\bibitem{Car07}
G.~Carron.
\newblock Riesz transforms on connected sums.
\newblock {\em Ann. Inst. Fourier (Grenoble)}, 57(7):2329--2343, 2007.
\newblock Festival Yves Colin de Verdi{\`e}re.

\bibitem{CCH06}
G.~Carron, T.~Coulhon, and A.~Hassell.
\newblock Riesz transform and {$L^p$}-cohomology for manifolds with {E}uclidean
  ends.
\newblock {\em Duke Math. J.}, 133(1):59--93, 2006.

\bibitem{CCFR17}
L.~Chen, T.~Coulhon, J.~Feneuil, and E.~Russ.
\newblock Riesz transform for {$1\le p\le 2$} without {G}aussian heat kernel
  bound.
\newblock {\em J. Geom. Anal.}, 27(2):1489--1514, 2017.

\bibitem{CMO15}
P.~Chen, J.~Magniez, and E.~M. Ouhabaz.
\newblock The {H}odge--de {R}ham {L}aplacian and {$L^p$}-boundedness of {R}iesz
  transforms on non-compact manifolds.
\newblock {\em Nonlinear Anal.}, 125:78--98, 2015.

\bibitem{CD99}
T.~Coulhon and X.~T. Duong.
\newblock Riesz transforms for {$1\leq p\leq 2$}.
\newblock {\em Trans. Amer. Math. Soc.}, 351(3):1151--1169, 1999.

\bibitem{CD03}
T.~Coulhon and X.~T. Duong.
\newblock Riesz transform and related inequalities on noncompact {R}iemannian
  manifolds.
\newblock {\em Comm. Pure Appl. Math.}, 56(12):1728--1751, 2003.

\bibitem{CMZ96}
T.~Coulhon, D.~M{\"u}ller, and J.~Zienkiewicz.
\newblock About {R}iesz transforms on the {H}eisenberg groups.
\newblock {\em Math. Ann.}, 305(2):369--379, 1996.

\bibitem{CS01}
M.~Cowling and A.~Sikora.
\newblock A spectral multiplier theorem for a sublaplacian on {$\rm SU(2)$}.
\newblock {\em Math. Z.}, 238(1):1--36, 2001.

\bibitem{Dev14}
B.~Devyver.
\newblock A {G}aussian estimate for the heat kernel on differential forms and
  application to the {R}iesz transform.
\newblock {\em Math. Ann.}, 358(1-2):25--68, 2014.

\bibitem{GS15}
C.~Guillarmou and D.~A. Sher.
\newblock Low energy resolvent for the {H}odge {L}aplacian: applications to
  {R}iesz transform, {S}obolev estimates, and analytic torsion.
\newblock {\em Int. Math. Res. Not. IMRN}, (15):6136--6210, 2015.

\bibitem{IM96}
T.~Iwaniec and G.~Martin.
\newblock Riesz transforms and related singular integrals.
\newblock {\em J. Reine Angew. Math.}, 473:25--57, 1996.

\bibitem{MR580107}
E.~Lenglart, D.~L\'epingle, and M.~Pratelli.
\newblock Pr\'esentation unifi\'ee de certaines in\'egalit\'es de la th\'eorie
  des martingales.
\newblock In {\em Seminar on {P}robability, {XIV} ({P}aris, 1978/1979)
  ({F}rench)}, volume 784 of {\em Lecture Notes in Math.}, pages 26--52.
  Springer, Berlin, 1980.
\newblock With an appendix by Lenglart.

\bibitem{LiHQ99}
H.-Q. Li.
\newblock La transformation de {R}iesz sur les vari{\'e}t{\'e}s coniques.
\newblock {\em J. Funct. Anal.}, 168(1):145--238, 1999.

\bibitem{LiHQ03II}
H.-Q. Li.
\newblock Analyse sur les vari{\'e}t{\'e}s cuspidales.
\newblock {\em Math. Ann.}, 326(4):625--647, 2003.

\bibitem{LiHQ03I}
H.-Q. Li and N.~Lohou{\'e}.
\newblock Transform{\'e}es de {R}iesz sur une classe de vari{\'e}t{\'e}s {\`a}
  singularit{\'e}s coniques.
\newblock {\em J. Math. Pures Appl. (9)}, 82(3):275--312, 2003.

\bibitem{LiXD08}
X.-D. Li.
\newblock Martingale transforms and {$L^p$}-norm estimates of {R}iesz
  transforms on complete {R}iemannian manifolds.
\newblock {\em Probab. Theory Related Fields}, 141(1-2):247--281, 2008.

\bibitem{LiXD10II}
X.-D. Li.
\newblock {$L^p$}-estimates and existence theorems for the
  {$\overline\partial$}-operator on complete {K}{\"a}hler manifolds.
\newblock {\em Adv. Math.}, 224(2):620--647, 2010.

\bibitem{LiXD10I}
X.-D. Li.
\newblock Riesz transforms on forms and {$L^p$}-{H}odge decomposition on
  complete {R}iemannian manifolds.
\newblock {\em Rev. Mat. Iberoam.}, 26(2):481--528, 2010.

\bibitem{LiXD14II}
X.-D. Li.
\newblock Erratum to: {M}artingale transforms and {$L^p$}-norm estimates of
  {R}iesz transforms on complete {R}iemannian manifolds [mr2372971].
\newblock {\em Probab. Theory Related Fields}, 159(1-2):405--408, 2014.

\bibitem{LiXD14I}
X.-D. Li.
\newblock Erratum to ``{R}iesz transforms on forms and {$L^p$}-{H}odge
  decomposition on complete {R}iemanian manifolds'' [mr2677005].
\newblock {\em Rev. Mat. Iberoam.}, 30(1):369--370, 2014.

\bibitem{Mey76}
P.~A. Meyer.
\newblock D{\'e}monstration probabiliste de certaines in{\'e}galit{\'e}s de
  {L}ittlewood-{P}aley. {I}. {L}es in{\'e}galit{\'e}s classiques.
\newblock pages 125--141. Lecture Notes in Math., Vol. 511, 1976.

\bibitem{Mey84}
P.-A. Meyer.
\newblock Transformations de {R}iesz pour les lois gaussiennes.
\newblock In {\em Seminar on probability, {XVIII}}, volume 1059 of {\em Lecture
  Notes in Math.}, pages 179--193. Springer, Berlin, 1984.

\bibitem{MPR15}
D.~M{\"u}ller, M.~M. Peloso, and F.~Ricci.
\newblock Analysis of the {H}odge {L}aplacian on the {H}eisenberg group.
\newblock {\em Mem. Amer. Math. Soc.}, 233(1095):vi+91, 2015.

\bibitem{MR0312140}
S.~K. Pichorides.
\newblock On the best values of the constants in the theorems of {M}. {R}iesz,
  {Z}ygmund and {K}olmogorov.
\newblock {\em Studia Math.}, 44:165--179. (errata insert), 1972.
\newblock Collection of articles honoring the completion by Antoni Zygmund of
  50 years of scientific activity, II.

\bibitem{RT16}
L.~Roncal and S.~Thangavelu.
\newblock Hardy's inequality for fractional powers of the sublaplacian on the
  {H}eisenberg group.
\newblock {\em Adv. Math.}, 302:106--158, 2016.

\bibitem{Ros}
S.~Rosenberg.
\newblock {\em The {L}aplacian on a {R}iemannian manifold}, volume~31 of {\em
  London Mathematical Society Student Texts}.
\newblock Cambridge University Press, Cambridge, 1997.
\newblock An introduction to analysis on manifolds.

\bibitem{Ste70}
E.~M. Stein.
\newblock {\em Topics in harmonic analysis related to the {L}ittlewood-{P}aley
  theory}.
\newblock Annals of Mathematics Studies, No. 63. Princeton University Press,
  Princeton, N.J.; University of Tokyo Press, Tokyo, 1970.

\bibitem{Str83}
R.~S. Strichartz.
\newblock Analysis of the {L}aplacian on the complete {R}iemannian manifold.
\newblock {\em J. Funct. Anal.}, 52(1):48--79, 1983.

\bibitem{vanNV17}
J.~van Neerven and R.~Versendaal.
\newblock {$L^p$}-analysis of the {H}odge-{D}irac operator associated with
  {W}itten {L}aplacians on complete {R}iemannian manifolds.
\newblock {\em J. Geom. Anal.}, 2017.

\end{thebibliography}
\end{document}